\documentclass[12pt]{article}

\usepackage[utf8]{inputenc}
\usepackage[T1]{fontenc}

\usepackage{amsmath}
\usepackage{amssymb}
\usepackage{amsthm}
\usepackage{mathtools}
\usepackage{mlmodern}

\usepackage{tikz}
\usetikzlibrary{arrows.meta,calc,decorations.pathmorphing}

\usepackage[pdfusetitle]{hyperref}  
\usepackage[all]{hypcap}  

\numberwithin{equation}{section}
\newtheorem{theorem}{Theorem}[section]
\newtheorem{lemma}[theorem]{Lemma}
\newtheorem{corollary}[theorem]{Corollary}
\theoremstyle{definition}

\DeclarePairedDelimiter{\abs}{|}{|}
\DeclarePairedDelimiter{\set}{\{}{\}}
\DeclarePairedDelimiter{\paren}{(}{)}
\DeclarePairedDelimiter{\gen}{\langle}{\rangle}

\newcommand{\CC}{\mathbb{C}}
\newcommand{\SC}{\mathfrak{S}}

\DeclareMathOperator{\Fun}{Fun}
\DeclareMathOperator{\id}{id}
\DeclareMathOperator{\pt}{pt}
\DeclareMathOperator{\csf}{CSF}

\newcommand{\arxiv}[1]{\href{https://arxiv.org/abs/#1}{arXiv:#1}}
\newcommand{\doi}[1]{\href{https://doi.org/#1}{doi:\detokenize{#1}}}
\newcommand{\mailto}[1]{\href{mailto:#1}{\texttt{#1}}}

\tikzset{
backdrop/.style={
  black!10,
  line width=7pt,
  line join=round,
  line cap=round,
},
vertex/.style={
  circle,
  fill=black,
  inner sep=0pt,
  outer sep=auto,
  minimum size=3pt,
},
edge12/.style={
  draw,
  semithick,
},
edge23/.style={
  draw,
  double=black!10,
  semithick,
},
edge13/.style={
  draw,
  decorate,
  decoration={
    snake,
    amplitude=1pt,
    segment length=6pt,
  },
  semithick,
},
arrow12/.style={
  edge12,
  <->,
  shorten <=1pt,
  shorten >=1pt,
},
arrow23/.style={
  edge23,
  {<[length=3pt]}-{>[length=3pt]},
  shorten <=1pt,
  shorten >=1pt,
},
arrow13/.style={
  edge13,
  <->,
  shorten <=1pt,
  shorten >=1pt,
  decoration={
    pre length=3pt,
    post length=3pt,
  },
},
pastille/.style={
  circle,
  inner sep=1pt,
  fill=white,
  text height=1.5ex,
},
partial1/.style={
  draw,
  |->,
  edge node=node [auto] {$\partial_1$},
},
partial2/.style={
  draw,
  |->,
  edge node=node [auto] {$\partial_2$},
},
base1/.pic={
  \draw[backdrop]
  (0,-.5) node [vertex] (id) {} --
  (0,.5)  node [vertex] (s1) {};
  \path[every edge/.style={edge12}]
  (id) edge (s1);
},
base2/.pic={
  \filldraw[backdrop]
  (0,-.707) node [vertex] (id)   {} --
  (.707,0)  node [vertex] (s3)   {} --
  (0,.707)  node [vertex] (s1s3) {} --
  (-.707,0) node [vertex] (s1)   {} -- cycle;
  \path[every edge/.style={edge12}]
  (id) edge (s1)  (s3) edge (s1s3);
  \path[every edge/.style={edge23}]
  (id) edge (s3)  (s1) edge (s1s3);
},
base3/.pic={
  \filldraw[backdrop]
  (-90:1) node [vertex] (id)   {} --
  (-30:1) node [vertex] (s2)   {} --
   (30:1) node [vertex] (s2s1) {} --
   (90:1) node [vertex] (w0)   {} --
  (150:1) node [vertex] (s1s2) {} --
  (210:1) node [vertex] (s1)   {} -- cycle;
},
cond12/.pic={
  \path[every edge/.style={edge12}]
  (id) edge (s1)  (s2) edge (s2s1)  (s1s2) edge (w0);
},
cond23/.pic={
  \path[every edge/.style={edge23}]
  (id) edge (s2)  (s1) edge (s1s2)  (s2s1) edge (w0);
},
cond13/.pic={
  \path[every edge/.style={edge13}]
(id) edge (w0)  (s1) edge (s2s1)  (s2) edge (s1s2);
},
}

\begin{document}
\title{Divided difference operators for Hessenberg representations}
\author{Mathieu Guay-Paquet}
\date{July 2025}
\maketitle

\begin{abstract}
The equivariant cohomology ring of a regular semisimple Hessenberg variety in type A is a free module over the equivariant cohomology ring of a point.
When equipped with Tymoczko's dot action, it becomes a twisted representation of the symmetric group, and the character of this representation is given by the chromatic quasisymmetric function of an indifference graph.
In this note, we use divided difference operators to decompose this representation as a direct sum of sub-representations in a way that categorifies the modular relation between chromatic quasisymmetric functions.
\end{abstract}

\section{Introduction}

For context, we outline the same story about cohomology rings of flag varieties three times, with increasing refinements, after defining some notation.

\bigskip{\small\par\noindent\textbf{Acknowledgments.}
I would like to thank Nicolas England, Franco Saliola, and Julianna Tymoczko for reviewing earlier drafts of this paper;
Megumi Harada and the organizers of the workshops ``Interactions between Hessenberg varieties, chromatic functions, and LLT polynomials'' at BIRS in 2022 and ``New perspectives of coinvariant rings from viewpoints of geometry and algebraic combinatorics'' at RIMS in 2025 for inviting me to speak about it;
and the members of LACIM for providing a welcoming and fruitful working environment.}

\subsection{Notation}

Fix a natural number $n$.
Let $S_n$ be the symmetric group of order $n$, viewed as the set of functions from $\set{1, \ldots, n}$ to itself, so that
\begin{equation}
  (vw)(i) = v\paren[\big]{w(i)}
  \quad\text{for $v, w \in S_n$ and $i \in \set{1, \ldots, n}$.}
\end{equation}
The group $S_n$ is generated by the \emph{adjacent transpositions},
\begin{equation}
  s_i = (i \leftrightarrow i+1)
  \quad\text{for $i \in \set{1, \ldots, n-1}$},
\end{equation}
which satisfy $s_i^2 = \id$ and the \emph{braid relations}:
\begin{equation}\begin{gathered}
  s_i s_{(i+1)} s_i = s_{(i+1)} s_i s_{(i+1)}
  \qquad\text{for $i \in \set{1, \ldots, n-2}$}, \\
  s_i s_k = s_k s_i
  \qquad\text{for $\abs{i - k} > 1$}.
\end{gathered}\end{equation}
We sometimes write a permutation in one-line notation, as
\begin{equation}
  w = [w(1), \ldots, w(n)].
\end{equation}
In particular, the \emph{longest permutation} is the permutation
\begin{equation}
  w_0 = [n, \ldots, 1].
\end{equation}

A \emph{Hessenberg function} is a function $h : \set{1, \ldots, n} \to \set{1, \ldots, n}$ such that
\begin{equation}\begin{gathered}
  i \leq h(i) \qquad\text{for $i \in \set{1, \ldots, n}$}, \\
  h(i) \leq h(i+1) \qquad\text{for $i \in \set{1, \ldots, n-1}$}.
\end{gathered}\end{equation}
Like permutations, we sometimes write them in one-line notation, as
\begin{equation}
  h = [h(1), \ldots, h(n)].
\end{equation}

We write $\Lambda$ for the ring of symmetric polynomials in $n$ variables, and $\Lambda_0$ for the ideal of symmetric polynomials with no constant term:
\begin{equation}
  \Lambda_0 = \set{f \in \Lambda \mid f(0, \ldots, 0) = 0}.
\end{equation}

We are interested in two commuting $S_n$-actions on the polynomial ring
\begin{equation}
  R = \CC[t_1, \ldots, t_n; x_1, \ldots, x_n]
\end{equation}
by $\CC$-algebra morphisms: the \emph{dot} action, on the left, where $w \in S_n$ permutes the $t_i$ variables and fixes the $x_i$ variables:
\begin{equation}
  w \cdot t_i = t_{w(i)}, \qquad
  w \cdot x_i = x_i,
\end{equation}
and the \emph{star} action, on the right, where $w \in S_n$ fixes the $t_i$ variables and permutes the $w_i$ variables:
\begin{equation}
  t_i * w = t_i, \qquad
  x_i * w = x_{w^{-1}(i)}.
\end{equation}
These two actions are named following Tymoczko~\cite{tymoczko07}, who considered the corresponding actions on the ring of functions
\begin{equation}
  H = \Fun(S_n, \CC[t_1, \ldots, t_n]),
\end{equation}
defined for $w \in S_n$ and $f \in H$ by
\begin{equation}\begin{aligned}
  (w \cdot f)(v; t_1, \ldots, t_n) &= f(w^{-1}v; t_{w(1)}, \ldots, t_{w(n)}), \\
  (f * w)(v; t_1, \ldots, t_n) &= f(vw^{-1}; t_1, \ldots, t_n).
\end{aligned}\end{equation}
The correspondence is given by the following $\CC$-algebra morphism, which respects both actions:
\begin{equation}\begin{aligned}
  \varphi : R &\to H \\
  t_i &\mapsto (v \mapsto t_i) \\
  x_i &\mapsto (v \mapsto t_{v(i)}).
\end{aligned}\end{equation}
By abuse of notation, we will write $t_i$ and $x_i$ for $\varphi(t_i)$ and $\varphi(x_i)$ in $H$.

For each adjacent transposition $s_i$, there is a corresponding \emph{divided difference operator} $\partial_i$ defined by
\begin{equation}
  \partial_i(f) = \frac{f * (1 - s_i)}{x_i - x_{(i+1)}}.
\end{equation}
This definition makes sense in the ring $R$, where the quotient always exists as a (unique) polynomial.
However, it only makes partial sense in the ring $H$, where the quotient may not exist (but is unique if it does exist).
One of our contributions in this paper is to detail some suitable domains and codomains for $\partial_i$ on $H$ (see \autoref{sec:domain-codomain}).
For now, note that, on the appropriate domains, the divided difference operators satisfy $\partial_i^2 = 0$ and the braid relations:
\begin{equation}\begin{gathered}
  \partial_i \partial_{(i+1)} \partial_i = \partial_{(i+1)} \partial_i \partial_{(i+1)}
  \qquad\text{for $i \in \set{1, \ldots, n-2}$}, \\
  \partial_i \partial_k = \partial_k \partial_i
  \qquad\text{for $\abs{i - k} > 1$},
\end{gathered}\end{equation}
since
\begin{multline}
  \partial_i \partial_{(i+1)} \partial_i(f) = \partial_{(i+1)} \partial_i \partial_{(i+1)}(f) \\
  = \frac{f * (1 - s_i - s_{(i+1)} + s_i s_{(i+1)} + s_{(i+1)} s_i - s_i s_{(i+1)} s_i)}{(x_i - x_{(i+1)}) (x_i - x_{(i+2)}) (x_{(i+1)} - x_{(i+2)})}
\end{multline}
and
\begin{equation}
\partial_i \partial_k(f) = \partial_k \partial_i(f)
= \frac{f * (1 - s_i - s_k + s_i s_k)}{(x_i - x_{(i+1)}) (x_k - x_{(k+1)})}.
\end{equation}
The divided difference operators are not ring homomorphisms, but they are $\CC$-linear and satisfy a (skew) product rule:
\begin{equation}
  \partial_i(fg) = \partial_i(f) g + (f * s_i) \partial_i(g).
\end{equation}

\subsection{First pass: ordinary cohomology}

Consider the variety of complete flags in $\CC^n$,
\begin{equation}
  X = \set[\big]{V_\bullet = (V_0 \subseteq \cdots \subseteq V_n) \bigm| \text{$\dim(V_i) = i$ for $i \in \set{0, \ldots, n}$}}.
\end{equation}
The Borel presentation for its ordinary cohomology ring is the polynomial ring quotient
\begin{equation}
  H^*(X) = \frac{\CC[x_1, \ldots, x_n]}{\gen{f(x_1, \ldots, x_n) \mid f \in \Lambda_0}},
\end{equation}
where $x_i$ is the first Chern class of the line bundle $V_i / V_{(i-1)}$ over $X$ (see, for example, \cite{kaji15} and references therein).

This ring has a very nice $\CC$-linear basis, given by the Schubert polynomials (for their history, see~\cite{lascoux95}).
These polynomials are indexed by the permutations in $S_n$, and they can be defined recursively by starting at the longest permutation, and applying divided difference operators to obtain the others:
\begin{equation}\begin{gathered}
  \SC_{w_0} = x_1^{n-1} x_2^{n-2} \cdots x_n^0, \\
  \SC_{w s_i} = \partial_i(\SC_w) \qquad\text{if $w(i) > w(i+1)$}.
\end{gathered}\end{equation}
Note that the divided difference operators are well-defined on $H^*(X)$ because they respect the ideal $\gen{f(x_1, \ldots, x_n) \mid f \in \Lambda_0}$ by the product rule; and that the indexing by permutations is well-defined because of the braid relations.

\subsection{Second pass: equivariant cohomology}

Consider the $n$-dimensional complex torus $T = (\CC^\times)^n$, whose classifying space is $BT = (\CC P^\infty)^n$.
The equivariant cohomology ring of a single point, equipped with the trivial $T$-action, is the polynomial ring
\begin{equation}
  H_T^*(\pt) = H^*(BT) = \CC[t_1, \ldots, t_n],
\end{equation}
where $t_i$ it the first Chern class of the tautological line bundle over the $i$th factor $\CC P^\infty$ of $BT$ (see, for example, \cite{tymoczko05} and references therein).

The flag variety $X$ can be equipped with a $T$-action by identifying $T$ with the group of invertible diagonal matrices on $\CC^n$.
The Borel presentation for its equivariant cohomology ring is the polynomial ring quotient
\begin{equation}\label{eq:equivariant-borel}
  H_T^*(X) = \frac{\CC[t_1, \ldots, t_n; x_1, \ldots, x_n]}{\gen{f(t_1, \ldots, t_n) - f(x_1, \ldots, x_x) \mid f \in \Lambda}}.
\end{equation}

We can recover the ordinary cohomology by setting $t_1 = \cdots = t_n = 0$.
The divided difference operators respect the quotienting ideal in~(\ref*{eq:equivariant-borel}), and we can lift the $\CC$-linear basis of Schubert polynomials for $H^*(X)$ to a $H_T^*(\pt)$-module basis of \emph{double} Schubert polynomials for $H_T^*(X)$, as follows:
\begin{equation}\begin{gathered}
  \SC'_{w_0} = \prod_{i+k \leq n} (x_i - t_k), \\
  \SC'_{w s_i} = \partial_i(\SC'_w) \qquad\text{if $w(i) > w(i+1)$}.
\end{gathered}\end{equation}
This basis shows that $H_T^*(X)$ is a free module over $\CC[t_1, \ldots, t_n]$, but in fact the same basis shows that $H_T^*(X)$ is also a free module over $\CC[x_1, \ldots, x_n]$.

As noted in~\cite{tymoczko07}, the dot action acts trivially on the subring $\CC[x_1, \ldots, x_n]$, so we can consider $H_T^*(X)$ as a representation of $S_n$ over $\CC[x_1, \ldots, x_n]$; this representation is a single copy of the regular representation of $S_n$.

The dot action acts non-trivially on $\CC[t_1, \ldots, t_n]$, but it does send this subring to itself, so we can consider $H_T^*(X)$ as a \emph{twisted} representation of $S_n$ over $\CC[t_1, \ldots, t_n]$; this representation consists of $n!$ copies of the (twisted) trivial representation of $S_n$.

\subsection{Third pass: Hessenberg varieties}

Now we consider some sub-varieties of the flag variety $X$.
Fix a diagonal matrix $M$ with distinct eigenvalues.
The specific choice does not matter for this story, only that:
\begin{itemize}
  \item
    the Jordan blocks of $M$ all have size one (so that $M$ is semisimple),
  \item
    the Jordan blocks have distinct eigenvalues (so that $M$ is regular),
  \item
    $M$ commutes with the torus $T$.
\end{itemize}
Then, for every Hessenberg function $h$, we have the corresponding (regular semisimple) Hessenberg variety
\begin{equation}
  X(h) = \set[\big]{V_\bullet \in X \bigm| \text{$MV_i \subseteq V_{h(i)}$ for $i \in \set{1, \ldots, n}$}}.
\end{equation}
As two extreme examples, note that $X([n, \ldots, n])$ is the entire flag variety, and that $X([1, \ldots, n])$ consists of just $n!$ isolated points, the permutation flags, which are the fixed points of the $T$-action on $X$.

For each Hessenberg function, the sub-variety $X(h) \subseteq X$ is preserved by the $T$-action, so we can consider its equivariant cohomology ring.
Instead of the Borel presentation as a quotient of the ring $R$, we have the GKM presentation~\cite{tymoczko07}, which exhibits $H_T^*(X(h))$ as a subring of $H$ (see \autoref{sec:divisibility-conditions}).
However, note that the kernel of the map $\varphi : R \to H$ is precisely the quotienting ideal in~\eqref{eq:equivariant-borel}, so the two presentations are closely related.

The concept which generalizes the basis of double Schubert polynomials is the concept of a \emph{flow-up} basis~\cite{harada-tymoczko17}. These are bases
\begin{equation}
  \set{f_w \mid w \in S_n} \subseteq H_T^*(X(h)),
\end{equation}
over $H_T^*(\pt)$, indexed by permutations, where
\begin{equation}
  f_w(v) = 0 \quad\text{if $w \not\leq v$ in Bruhat order},
\end{equation}
and where $f_w(w)$ is a prescribed product of factors of the form $(t_i - t_k)$.
In the case of the entire flag variety, the double Schubert polynomials are the unique flow-up basis, but in general the flow-up bases are not unique.
Still, flow-up bases do exist, and they show that $H_T^*(X(h))$ is a free module over $\CC[t_1, \ldots, t_n]$, and a free module over $\CC[x_1, \ldots, x_n]$.

The dot action on $H$ preserves the subring $H_T^*(X(h))$.
When we consider it as a representation of $S_n$ over $\CC[x_1, \ldots, x_n]$, its graded Frobenius characteristic is a unicellular LLT polynomial.
When we consider it as a twisted representation of $S_n$ over $\CC[t_1, \ldots, t_n]$, its graded Frobenius characteristic is a chromatic quasisymmetric function (up to the fundamental involution on symmetric functions)~\cite{brosnan-chow18,carlsson-mellit18,guay-paquet16,lascoux-leclerc-thibon97,shareshian-wachs12,shareshian-wachs16,stanley95}.

Both unicellular LLT polynomials and chromatic quasisymmetric functions are known to satisfy the modular relation (a variant on the deletion-contraction recurrence for chromatic polynomials of graphs), which states that, for certain triples $(h_{-}, h, h_{+})$ of Hessenberg functions:
\begin{equation}\label{eq:modular-csf}
  (1+q) \csf_q(h) = \csf_q(h_{+}) + q \csf_q(h_{-})
\end{equation}
In fact, together with a relatively small set of base cases, the modular relation uniquely determines these two families of symmetric functions~\cite{abreu-nigro21,guay-paquet13,orellana-scott14}.

Another contribution of this paper is to give a decomposition of the dot representation on $H_T^*(X(h))$ which corresponds to the modular relation, by using divided difference operators (see \autoref{thm:stable} and \autoref{thm:almost-stable}).

Note that Horiguchi, Masuda and Sato~\cite{horiguchi-masuda-sato24} have also given such a decomposition, which is nicely motivated by the geometry of Hessenberg varieties.
They exhibit the blow-up of $X(h_{+})$ along $X(h_{-})$ as a $\CC P^1$-bundle over $X(h)$ equipped with a $T$-action.
Based on this, they identify a certain subring of
\begin{equation}
  H_T^*(X(h)) \oplus \paren[\big]{H_T^*(X(h)) * s_i}
\end{equation}
as being isomorphic to both:
\begin{itemize}
  \item
    the direct sum of a plain copy of $H_T^*(X(h))$ and a degree-shifted copy of $H_T^*(X(h))$, corresponding to the left-hand side of~\eqref{eq:modular-csf}, and
  \item
    the direct sum of a plain copy of $H_T^*(X(h_{+}))$ and a degree-shifted copy of $H_T^*(X(h_{-}))$, corresponding to the right-hand side.
\end{itemize}

In contrast, our decomposition is only algebraic, and corresponds to a version of~\eqref{eq:modular-csf} which is divided by $(1+q)$.
Using the divided difference operators, we identify a sub-representation
\begin{equation}
  H_T^*(X(h_{+}))^{*s_i} \subseteq H_T^*(X(h_{+}))
\end{equation}
such that $H_T^*(X(h_{+}))$ is the direct sum of a plain copy and a degree-shifted copy of $H_T^*(X(h_{+}))^{*s_i}$; similarly for $H_T^*(X(h_{-}))$.
Then, we show that a single, plain copy of $H_T^*(X(h))$ is the direct sum of a plain copy of $H_T^*(X(h_{+}))^{*s_i}$ and a degree-shifted copy of $H_T^*(X(h_{-}))^{*s_i}$.

\section{Divisibility conditions}\label{sec:divisibility-conditions}

We now describe a family of subrings of the ring of functions $H$, which will be a convenient setting for the divided difference operators, and among which are the equivariant cohomology rings of Hessenberg varieties.

For an element $f \in H$ and a transposition $\tau = (i \leftrightarrow k) \in S_n$, we say that \emph{$f$ satisfies condition $\tau$} if
\begin{equation}
  f * (1 - \tau) \quad\text{is a multiple of}\quad (x_i - x_k).
\end{equation}
Note that, in particular, any multiple of $(x_i - x_k)$ satisfies condition $\tau$:
\begin{equation}
  \paren[\big]{(x_i - x_k) f} * (1 - \tau) = (x_i - x_k) \paren[\big]{f * (1 + \tau)}.
\end{equation}
The set of elements which satisfy condition $\tau$ is closed under addition and multiplication, so they form a subring of $H$, which we call $H_\tau$.
More generally, for a set of transpositions $C \subseteq S_n$, we write
\begin{equation}
  H_C = \bigcap_{\tau \in C} H_\tau
\end{equation}
for the subring of elements which satisfy all conditions in $C$.
If $C$ is the empty set, then $H_C$ is the entire ring $H$.

Note that the elements $t_1, \ldots, t_n$ and $x_1, \ldots, x_n$ satisfy all divisibility conditions, so the entire image of $\varphi : R \to H$ does too.

If $w \in S_n$ and $f$ satisfies condition $\tau$, then
\begin{equation}\begin{gathered}
  \text{$w \cdot f$ satisfies condition $\tau$, and} \\
  \text{$f * w$ satisfies condition $w^{-1}\tau w$,}
\end{gathered}\end{equation}
so each $H_C$ is closed under the dot action, but usually \emph{not} the star action.

For a Hessenberg function $h$, the GKM presentation of the equivariant cohomology ring for the corresponding Hessenberg variety is
\begin{equation}
  H_T^*(X(h)) = H_{C(h)},
\end{equation}
with the following set of divisibility conditions:
\begin{equation}\label{eq:hessenberg-conditions}
  C(h) = \set{(i \leftrightarrow k) \mid i < k \leq h(i)}.
\end{equation}

Recall that the definition of the $i$th divided difference operator is
\begin{equation}
  \partial_i(f) = \frac{f * (1 - s_i)}{x_i - x_{(i+1)}},
\end{equation}
if the quotient exists in $H$.
Since this is the defining condition of the subring $H_{s_i}$, we can write the largest domain of definition of $\partial_i$ as:
\begin{equation}
  \partial_i : H_{s_i} \to H.
\end{equation}
For Hessenberg functions, note that $s_i \in C(h)$ whenever $h(i) \neq i$.
We give a more careful account of some domains and codomains for $\partial_i$ in \autoref{sec:domain-codomain}.

We record here some of the convenient computational properties of divided differences which follow directly from the definition.
\begin{lemma}\label{lem:compute}
  For $f, g \in H_{s_i}$ and $w \in S_n$, we have:
  \begin{enumerate}
    \item $\partial_i(f + g) = \partial_i(f) + \partial_i(g)$,
    \item $\partial_i(f g) = \partial_i(f) \, g + (f * s_i) \, \partial_i(g)$,
    \item $\partial_i(f g) = f \, \partial_i(g)$ when $\partial_i(f) = 0$,
    \item $\partial_i(f) = 0$ exactly when $f * s_i = f$,
    \item $\partial_i(w \cdot f) = w \cdot \partial_i(f)$,
    \item $\partial_i(f * s_i) = -\partial_i(f)$,
    \item $\partial_i(f) * s_i = \partial_i(f)$,
    \item $\partial_i(f) * w = \partial_k(f * w)$ when $w(k) = i$ and $w(k+1) = i+1$,
    \item $\partial_i(t_k) = 0$,
    \item $\partial_i$ is $\CC[t_1, \ldots, t_n]$-linear,
    \item $\partial_i(x_i) = 1$,
    \item $\partial_i(x_{(i+1)}) = -1$,
    \item $\partial_i(x_k) = 0$ when $k \notin \set{i, i+1}$.
  \end{enumerate}
\end{lemma}
\begin{proof}
  By direct computation, keeping in mind that if the quotient exists, then $(f/g) * w = (f*w)/(g*w)$.
\end{proof}

\section{Examples for \texorpdfstring{$n = 3$}{n = 3}}

For concreteness, here is a convenient way of writing down elements of $H$, at least for $n = 3$.
We will draw $S_3$ as a collection of 6 vertices, arranged as a hexagon, where each vertex represents a permutation, labelled as follows:
\begin{equation}\begin{tikzpicture}[baseline=0]
  \pic{base3};
  \node[below=2pt] at (id)   {$[123]$};
  \node[left=2pt]  at (s1)   {$s_1 = [213]$};
  \node[right=2pt] at (s2)   {$[132] = s_2$};
  \node[left=2pt]  at (s1s2) {$s_1 s_2 = [231]$};
  \node[right=2pt] at (s2s1) {$[312] = s_2 s_1$};
  \node[above=2pt] at (w0)   {$[321]$};
  \useasboundingbox (-4,0) (4,0);
\end{tikzpicture}\end{equation}
Then, an element of $H$ can be seen as an $S_n$-indexed tuple of multivariate polynomials, and we can write each polynomial next to the corresponding vertex in the hexagon. For example, the element of $H$ given by the function
\begin{equation}\begin{aligned}
  [123] &\mapsto 0 &\quad
  [213] &\mapsto 5t_1 &\quad
  [132] &\mapsto 1 \\
  [231] &\mapsto t_2 t_3 &
  [312] &\mapsto (t_1 + t_3) &
  [321] &\mapsto t_1^2 (t_2 - t_3)
\end{aligned}\end{equation}
can be drawn as
\begin{equation}\label{eq:arbitrary-element}\begin{tikzpicture}[baseline=0]
  \pic{base3};
  \node[below=2pt] at (id)   {$0$};
  \node[left=2pt]  at (s1)   {$5t_1$};
  \node[right=3pt] at (s2)   {$1$};
  \node[left=2pt]  at (s1s2) {$t_2 t_3$};
  \node[right=2pt] at (s2s1) {$(t_1 + t_3)$};
  \node[above=2pt] at (w0)   {$t_1^2 (t_2 - t_3)$};
  \useasboundingbox (-4,0) (4,0);
\end{tikzpicture}\end{equation}

\pagebreak

The star action of $S_3$ on $H$ on the right permutes the polynomials attached to the vertices.
For the three transpositions $s_1, s_2, w_0 \in S_3$, the effect is:
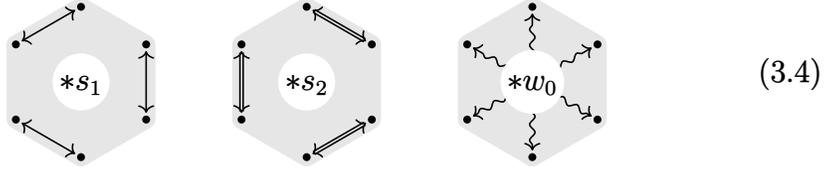
\begin{equation}\begin{tikzpicture}[baseline=0]
  \begin{scope}[shift={(-3,0)}]
    \pic{base3};
    \path[every edge/.style={arrow12}]
    (id) edge (s1)  (s2) edge (s2s1)  (s1s2) edge (w0);
    \node[pastille] at (0,0) {$*s_1$};
  \end{scope}
  \begin{scope}[shift={(0,0)}]
    \pic{base3};
    \path[every edge/.style={arrow23}]
    (id) edge (s2)  (s1) edge (s1s2)  (s2s1) edge (w0);
    \node[pastille] at (0,0) {$*s_2$};
  \end{scope}
  \begin{scope}[shift={(3,0)}]
    \pic{base3};
    \path[every edge/.style={arrow13}]
    (id) edge (w0)  (s1) edge (s2s1)  (s2) edge (s1s2);
    \node[pastille] at (0,0) {$*w_0$};
  \end{scope}
  \useasboundingbox (0,-1.3) (0,1.3);
\end{tikzpicture}\end{equation}
The three actions above are drawn with different kinds of arrows; each kind corresponds to a divisibility condition which could be imposed on the elements of $H$.
For the element $f \in H$ of~\eqref{eq:arbitrary-element}, the element $(f * s_1)$ is:
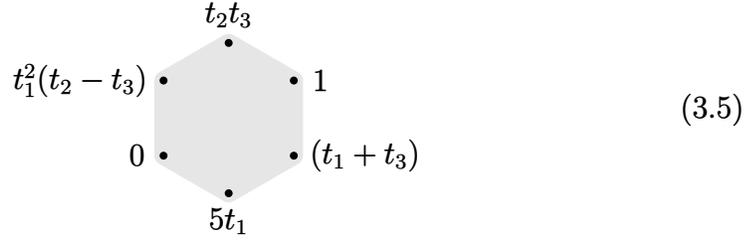
\begin{equation}\begin{tikzpicture}[baseline=0]
  \pic{base3};
  \node[below=2pt] at (id)   {$5t_1$};
  \node[left=3pt]  at (s1)   {$0$};
  \node[right=2pt] at (s2)   {$(t_1 + t_3)$};
  \node[left=2pt]  at (s1s2) {$t_1^2 (t_2 - t_3)$};
  \node[right=3pt] at (s2s1) {$1$};
  \node[above=2pt] at (w0)   {$t_2 t_3$};
  \useasboundingbox (-4,0) (4,0);
\end{tikzpicture}\end{equation}

The dot action of $S_3$ on $H$ on the left permutes the variables $t_1, t_2, t_3$ in addition to permuting the polynomials attached to the vertices. For the three transpositions $s_1, s_2, w_0 \in S_3$, the effect is:
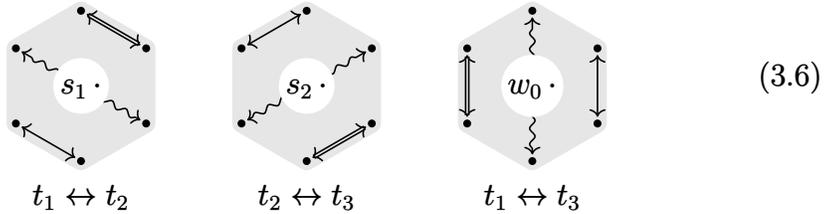
\begin{equation}\begin{tikzpicture}[baseline=0]
  \begin{scope}[shift={(-3,0)}]
    \pic{base3};
    \path
    (id)   edge[arrow12] (s1)
    (s2)   edge[arrow13] (s1s2)
    (s2s1) edge[arrow23] (w0);
    \node[pastille] at (0,0) {$s_1 \,\cdot$};
    \node[below=5pt] at (id) {$t_1 \leftrightarrow t_2$};
  \end{scope}
  \begin{scope}[shift={(0,0)}]
    \pic{base3};
    \path
    (id)   edge[arrow23] (s2)
    (s1)   edge[arrow13] (s2s1)
    (s1s2) edge[arrow12] (w0);
    \node[pastille] at (0,0) {$s_2 \,\cdot$};
    \node[below=5pt] at (id) {$t_2 \leftrightarrow t_3$};
  \end{scope}
  \begin{scope}[shift={(3,0)}]
    \pic{base3};
    \path
    (id) edge[arrow13] (w0)
    (s1) edge[arrow23] (s1s2)
    (s2) edge[arrow12] (s2s1);
    \node[pastille] at (0,0) {$w_0 \,\cdot$};
    \node[below=5pt] at (id) {$t_1 \leftrightarrow t_3$};
  \end{scope}
  \useasboundingbox (0,1.3);
\end{tikzpicture}\end{equation}
For the element $f \in H$ of~\eqref{eq:arbitrary-element}, the element $(s_1 \cdot f)$ is:
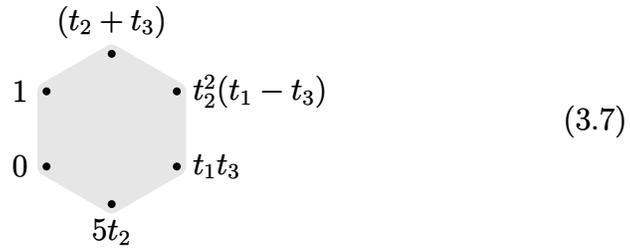
\begin{equation}\begin{tikzpicture}[baseline=0]
  \pic{base3};
  \node[below=2pt] at (id)   {$5t_2$};
  \node[left=3pt]  at (s1)   {$0$};
  \node[right=2pt] at (s2)   {$t_1 t_3$};
  \node[left=3pt]  at (s1s2) {$1$};
  \node[right=2pt] at (s2s1) {$t_2^2 (t_1 - t_3)$};
  \node[above=2pt] at (w0)   {$(t_2 + t_3)$};
  \useasboundingbox (-4,0) (4,0);
\end{tikzpicture}\end{equation}

The elements $t_1, t_2, t_3, x_1, x_2, x_3 \in H$, which satisfy all three divisibility conditions, are:
\begin{equation}\begin{tikzpicture}[baseline=0]
  \begin{scope}[shift={(0,1.8)}]
    \begin{scope}[shift={(-3.2,0)}]
      \pic{base3}; \pic{cond12}; \pic{cond23}; \pic{cond13};
      \node[below=2pt] at (id)   {$t_1$};
      \node[left=2pt]  at (s1)   {$t_1$};
      \node[right=2pt] at (s2)   {$t_1$};
      \node[left=2pt]  at (s1s2) {$t_1$};
      \node[right=2pt] at (s2s1) {$t_1$};
      \node[above=2pt] at (w0)   {$t_1$};
      \node[pastille] at (0,0) {$t_1$};
    \end{scope}
    \begin{scope}[shift={(0,0)}]
      \pic{base3}; \pic{cond12}; \pic{cond23}; \pic{cond13};
      \node[below=2pt] at (id)   {$t_2$};
      \node[left=2pt]  at (s1)   {$t_2$};
      \node[right=2pt] at (s2)   {$t_2$};
      \node[left=2pt]  at (s1s2) {$t_2$};
      \node[right=2pt] at (s2s1) {$t_2$};
      \node[above=2pt] at (w0)   {$t_2$};
      \node[pastille] at (0,0) {$t_2$};
    \end{scope}
    \begin{scope}[shift={(3.2,0)}]
      \pic{base3}; \pic{cond12}; \pic{cond23}; \pic{cond13};
      \node[below=2pt] at (id)   {$t_3$};
      \node[left=2pt]  at (s1)   {$t_3$};
      \node[right=2pt] at (s2)   {$t_3$};
      \node[left=2pt]  at (s1s2) {$t_3$};
      \node[right=2pt] at (s2s1) {$t_3$};
      \node[above=2pt] at (w0)   {$t_3$};
      \node[pastille] at (0,0) {$t_3$};
    \end{scope}
  \end{scope}
  \begin{scope}[shift={(0,-1.8)}]
    \begin{scope}[shift={(-3.2,0)}]
      \pic{base3}; \pic{cond12}; \pic{cond23}; \pic{cond13};
      \node[below=2pt] at (id)   {$t_1$};
      \node[left=2pt]  at (s1)   {$t_2$};
      \node[right=2pt] at (s2)   {$t_1$};
      \node[left=2pt]  at (s1s2) {$t_2$};
      \node[right=2pt] at (s2s1) {$t_3$};
      \node[above=2pt] at (w0)   {$t_3$};
      \node[pastille] at (0,0) {$x_1$};
    \end{scope}
    \begin{scope}[shift={(0,0)}]
      \pic{base3}; \pic{cond12}; \pic{cond23}; \pic{cond13};
      \node[below=2pt] at (id)   {$t_2$};
      \node[left=2pt]  at (s1)   {$t_1$};
      \node[right=2pt] at (s2)   {$t_3$};
      \node[left=2pt]  at (s1s2) {$t_3$};
      \node[right=2pt] at (s2s1) {$t_1$};
      \node[above=2pt] at (w0)   {$t_2$};
      \node[pastille] at (0,0) {$x_2$};
    \end{scope}
    \begin{scope}[shift={(3.2,0)}]
      \pic{base3}; \pic{cond12}; \pic{cond23}; \pic{cond13};
      \node[below=2pt] at (id)   {$t_3$};
      \node[left=2pt]  at (s1)   {$t_3$};
      \node[right=2pt] at (s2)   {$t_2$};
      \node[left=2pt]  at (s1s2) {$t_1$};
      \node[right=2pt] at (s2s1) {$t_2$};
      \node[above=2pt] at (w0)   {$t_1$};
      \node[pastille] at (0,0) {$x_3$};
    \end{scope}
  \end{scope}
\end{tikzpicture}\end{equation}

See \autoref{app:flow-up} for more examples.

\section{Domains and codomains}\label{sec:domain-codomain}

The divided difference operators aren't defined on all of $H$, and don't preserve the subrings $H_C$ in general.
But, motivated by the following lemma, we say that a set of conditions $C$ is \emph{$s_i$-stable} if
\begin{equation}
  s_i \in C
  \quad\text{and}\quad
  s_i C s_i = C.
\end{equation}

\begin{lemma}\label{lem:stability}
  If $C$ is $s_i$-stable, then $\partial_i$ is defined on $H_C$, and $\partial_i(H_C) \subseteq H_C$.
\end{lemma}

Since $H_C \subseteq H_D$ whenever $D \subseteq C$, we can also consider the largest $s_i$-stable subset of $C$ and phrase this equivalently as:

\begin{corollary}
  If $s_i \in C$ and $D = C \cap s_i C s_i$, then $\partial_i$ is defined on $H_C$, and $\partial_i(H_C) \subseteq H_D$.
\end{corollary}

\begin{proof}[Proof of \autoref*{lem:stability}]
Since $s_i \in C$, it's immediate that $\partial_i$ is defined on $H_C$.
To show that $\partial_i(H_C) \subseteq H_C$, we will reduce to the case of three specific sets $C$, each of which can be checked on a small $\CC[t_1, \ldots, t_n]$-submodule of $H_C$, and provide an explicit basis for this submodule.

As a first reduction, it suffices to check that $\partial_i(H_C) \subseteq H_C$ for the \emph{minimal} $s_i$-stable sets $C$, which are of the form
\begin{equation}
  C(i, \tau) = \set{s_i,\; \tau,\; s_i \tau s_i}
  \quad\text{for a transposition $\tau \in S_n$,}
\end{equation}
because for a general $s_i$-stable set $C$, we have
\begin{equation}
  H_C = \bigcap_{\tau \in C} H_{C(i, \tau)}.
\end{equation}

As a second reduction, it suffices to check the three specific $s_1$-stable sets
\begin{equation}\begin{aligned}
  C(1, s_1) &= \set{(1 \leftrightarrow 2)} \\
  C(1, s_2) &= \set{(1 \leftrightarrow 2), (1 \leftrightarrow 3), (2 \leftrightarrow 3)} \\
  C(1, s_3) &= \set{(1 \leftrightarrow 2), (3 \leftrightarrow 4)},
\end{aligned}\end{equation}
because each $C(i, \tau)$ can be reduced to one of them, depending on the case:
\begin{description}
  \item[Case $s_i = \tau$:]
    Let $w \in S_n$ be a permutation with
    \begin{equation}
      w(i) = 1, \qquad
      w(i+1) = 2.
    \end{equation}
    Then, using ``$\subseteq^?$'' to denote the inclusion we are reducing to, we have
    \begin{equation}\begin{aligned}
      \partial_i\paren[\big]{H_{C(i, \tau)}}
      = \partial_i\paren[\big]{H_{C(1, s_1)} * w}
      &= \partial_1\paren[\big]{H_{C(1, s_1)}} * w \\
      &\subseteq^? H_{C(1, s_1)} * w = H_{C(i, \tau)}.
    \end{aligned}\end{equation}

  \item[Case $s_i \neq \tau$ and $\tau \neq s_i \tau s_i$:]
    In this case, $\tau$ and $s_i \tau s_i$ must be of the form $(i \leftrightarrow k)$ and $(i+1 \leftrightarrow k)$, in some order, for some $k \notin \set{i, i+1}$.
    Let $w \in S_n$ be a permutation with
    \begin{equation}
      w(i) = 1, \qquad
      w(i+1) = 2, \qquad
      w(k) = 3.
    \end{equation}
    Then, we have
    \begin{equation}\begin{aligned}
      \partial_i\paren[\big]{H_{C(i, \tau)}}
      = \partial_i\paren[\big]{H_{C(1, s_2)} * w}
      &= \partial_1\paren[\big]{H_{C(1, s_2)}} * w \\
      &\subseteq^? H_{C(1, s_2)} * w = H_{C(i, \tau)}
    \end{aligned}\end{equation}

  \item[Case $s_i \neq \tau$ and $\tau = s_i \tau s_i$:]
    In this case, $\tau$ must be of the form $(j \leftrightarrow k)$ for some $\set{j, k}$ disjoint from $\set{i, i+1}$.
    Take $w \in S_n$ such that
    \begin{equation}
      w(i) = 1, \quad
      w(i+1) = 2, \quad
      w(j) = 3, \quad
      w(k) = 4.
    \end{equation}
    Then, we have
    \begin{equation}\begin{aligned}
      \partial_i\paren[\big]{H_{C(i, \tau)}}
      = \partial_i\paren[\big]{H_{C(1, s_3)} * w}
      &= \partial_1\paren[\big]{H_{C(1, s_3)}} * w \\
      &\subseteq^? H_{C(1, s_3)} * w = H_{C(i, \tau)}
    \end{aligned}\end{equation}
\end{description}

As an independent reduction, it suffices to check that $\partial_i(H_C) \subseteq H_C$ on a certain $\CC[t_1, \ldots, t_n]$-submodule of $H_C$.
Specifically, let $\gen{C}$ be the subgroup of $S_n$ generated by $C$, and let $1_{\gen{C}} \in H_C$ be defined by
\begin{equation}
  1_{\gen{C}}(v; t_1, \ldots, t_n) = \begin{cases}
    1 &\text{if $v \in \gen{C}$} \\
    0 &\text{otherwise.}
  \end{cases}
\end{equation}
Then, $1_{\gen{C}} H_C$ is the $\CC[t_1, \ldots, t_n]$-submodule of functions in $H_C$ which are zero outside of $\gen{C}$.
The unit element in $H_C$ decomposes as
\begin{equation}
  1 = \sum_w w \cdot 1_{\gen{C}},
\end{equation}
where the sum is over an arbitrary choice of representatives $w$ for the cosets $w \gen{C} \subseteq S_n$.
For an element $f \in H_C$ and a permutation $w \in S_n$, let
\begin{equation}
  f_w = 1_{\gen{C}} \, (w^{-1} \cdot f) \in 1_{\gen{C}} H_C.
\end{equation}
Then, the element $f$ can be decomposed as
\begin{equation}
  f = \sum_w w \cdot f_w,
\end{equation}
so that
\begin{equation}
  \partial_i(f) = \sum_w w \cdot \partial_i(f_w).
\end{equation}
Thus, $\partial_i(H_C) \subseteq H_C$ is implied by $\partial_i\paren[\big]{1_{\gen{C}} H_C} \subseteq H_C$.

Given all of these reductions, the remaining task is to compute $\partial_1$ on a $\CC[t_1, \ldots, t_n]$-linear basis of $1_{\gen{C}} H_C$ for $C \in \set{C(1, s_1),\; C(1, s_3),\; C(1, s_2)}$.
\begin{description}
  \item[Case $C = C(1, s_1)$:]
    The following two elements are a $\CC[t_1, \ldots, t_n]$-basis:
    \begin{equation}
      1_{\gen{C}} (x_1 - t_1), \qquad
      1_{\gen{C}}
    \end{equation}
    and the computation of $\partial_1$ on them is illustrated below:
    \begin{equation}\begin{tikzpicture}[baseline=0]
      \begin{scope}
        \pic{base1};
        \node[below=2pt] at (id) {$\id$};
        \node[above=2pt] at (s1) {$s_1$};
      \end{scope}
      \begin{scope}[shift={(2,0)}]
        \pic{base1};
        \node[below=2pt] at (id) {$0$};
        \node[above=2pt] at (s1) {$(t_2-t_1)$};
      \end{scope}
      \begin{scope}[shift={(4,0)}]
        \pic{base1};
        \node[below=2pt] at (id) {$1$};
        \node[above=2pt] at (s1) {$1$};
      \end{scope}
      \node at (5.9,0) {$0$};
      \path[every edge/.style=partial1,every node/.style=swap]
        (2.5,0) edge (3.5,0)
        (4.5,0) edge (5.5,0);
    \end{tikzpicture}\end{equation}

  \item[Case $C = C(1, s_3)$:]
    The following four elements are a $\CC[t_1, \ldots, t_n]$-basis:
    \begin{equation}\begin{aligned}
      & 1_{\gen{C}} (x_1 - t_1)(x_3 - t_3), &
      & 1_{\gen{C}} (x_3 - t_3), \\
      & 1_{\gen{C}} (x_1 - t_1), &
      & 1_{\gen{C}},
    \end{aligned}\end{equation}
    and the computation of $\partial_1$ on these elements is illustrated below, using the abbreviation $t_{ik} = (t_i - t_k)$:
    \begin{equation}\begin{tikzpicture}[baseline=0]
      \begin{scope}[shift={(2.2,0)}]
        \pic{base2};
        \node[below=2pt] at (id)   {$\id$};
        \node[left=1pt]  at (s1)   {$s_1$};
        \node[right=1pt] at (s3)   {$s_3$};
        \node[above=2pt] at (s1s3) {$s_1 s_3$};
      \end{scope}
      \begin{scope}[shift={(0,-2.7)}]
        \begin{scope}[shift={(0,0)}]
          \pic{base2};
          \node[below=2pt,gray] at (id)   {$0$};
          \node[left=2pt,gray]  at (s1)   {$0$};
          \node[right=2pt,gray] at (s3)   {$0$};
          \node[above=2pt]      at (s1s3) {$t_{21} t_{43}$};
        \end{scope}
        \begin{scope}[shift={(4.4,0)}]
          \pic{base2};
          \node[below=2pt,gray] at (id)   {$0$};
          \node[left=2pt,gray]  at (s1)   {$0$};
          \node[right=1pt]      at (s3)   {$t_{43}$};
          \node[above=2pt]      at (s1s3) {$t_{43}$};
        \end{scope}
        \node at (8,0) {$0$};
        \path[every edge/.style=partial1,every node/.style=swap]
          (1.4,0) edge (3,0)
          (6,0) edge (7.6,0);
      \end{scope}
      \begin{scope}[shift={(0,-5.4)}]
        \begin{scope}[shift={(0,0)}]
          \pic{base2};
          \node[below=2pt,gray] at (id)   {$0$};
          \node[left=1pt]       at (s1)   {$t_{21}$};
          \node[right=2pt,gray] at (s3)   {$0$};
          \node[above=2pt]      at (s1s3) {$t_{21}$};
        \end{scope}
        \begin{scope}[shift={(4.4,0)}]
          \pic{base2};
          \node[below=2pt] at (id)   {$1$};
          \node[left=2pt]  at (s1)   {$1$};
          \node[right=2pt] at (s3)   {$1$};
          \node[above=2pt] at (s1s3) {$1$};
        \end{scope}
        \node at (8,0) {$0$};
        \path[every edge/.style=partial1,every node/.style=swap]
          (1.4,0) edge (3,0)
          (5.8,0) edge (7.6,0);
      \end{scope}
    \end{tikzpicture}\end{equation}

  \pagebreak

  \item[Case $C = C(1, s_2)$:]
    The following six elements are a $\CC[t_1, \ldots, t_n]$-basis:
    \begin{equation}\begin{aligned}
      & 1_{\gen{C}} (x_2 - t_1)(x_1 - t_1)(x_1 - t_2), &
      & 1_{\gen{C}} (x_2 - t_1)(x_1 - t_1), \\
      & 1_{\gen{C}} (x_1 - t_1)(x_1 - t_2), &
      & 1_{\gen{C}} (t_3 - x_3), \\
      & 1_{\gen{C}} (x_1 - t_1), &
      & 1_{\gen{C}},
    \end{aligned}\end{equation}
    and the computation of $\partial_1$ on these elements is illustrated below, again using the abbreviation $t_{ik} = (t_i - t_k)$:
    \begin{equation}\begin{tikzpicture}[baseline=0]
      \begin{scope}[shift={(2.5,0)}]
        \pic{base3}; \pic{cond12}; \pic{cond23}; \pic{cond13};
        \node[below=2pt] at (id)   {$\id$};
        \node[left=4pt]  at (s1)   {$s_1$};
        \node[right=4pt] at (s2)   {$s_2$};
        \node[left=0pt]  at (s1s2) {$s_1 s_2$};
        \node[right=0pt] at (s2s1) {$s_2 s_1$};
        \node[above=2pt] at (w0)   {$w_0$};
      \end{scope}
      \begin{scope}[shift={(0,-3.2)}]
        \begin{scope}[shift={(0,0)}]
          \pic{base3}; \pic{cond12}; \pic{cond23}; \pic{cond13};
          \node[below=2pt,gray] at (id)   {$0$};
          \node[left=2pt,gray]  at (s1)   {$0$};
          \node[right=2pt,gray] at (s2)   {$0$};
          \node[left=2pt,gray]  at (s1s2) {$0$};
          \node[right=2pt,gray] at (s2s1) {$0$};
          \node[above=2pt]      at (w0)   {$t_{21} t_{31} t_{32}$};
        \end{scope}
        \begin{scope}[shift={(5,0)}]
          \pic{base3}; \pic{cond12}; \pic{cond23}; \pic{cond13};
          \node[below=2pt,gray] at (id)   {$0$};
          \node[left=2pt,gray]  at (s1)   {$0$};
          \node[right=2pt,gray] at (s2)   {$0$};
          \node[left=0pt]       at (s1s2) {$t_{21} t_{31}$};
          \node[right=2pt,gray] at (s2s1) {$0$};
          \node[above=2pt]      at (w0)   {$t_{21} t_{31}$};
        \end{scope}
        \node at (8.8,0) {$0$};
        \path[every edge/.style=partial1,every node/.style=swap]
          (1.7,0) edge (3.3,0)
          (6.7,0) edge (8.3,0);
      \end{scope}
      \begin{scope}[shift={(0,-6.4)}]
        \begin{scope}[shift={(0,0)}]
          \pic{base3}; \pic{cond12}; \pic{cond23}; \pic{cond13};
          \node[below=2pt,gray] at (id)   {$0$};
          \node[left=2pt,gray]  at (s1)   {$0$};
          \node[right=2pt,gray] at (s2)   {$0$};
          \node[left=2pt,gray]  at (s1s2) {$0$};
          \node[right=0pt]      at (s2s1) {$t_{31} t_{32}$};
          \node[above=2pt]      at (w0)   {$t_{31} t_{32}$};
        \end{scope}
        \begin{scope}[shift={(5,0)}]
          \pic{base3}; \pic{cond12}; \pic{cond23}; \pic{cond13};
          \node[below=2pt,gray] at (id)   {$0$};
          \node[left=2pt,gray]  at (s1)   {$0$};
          \node[right=1pt]      at (s2)   {$t_{32}$};
          \node[left=1pt]       at (s1s2) {$t_{31}$};
          \node[right=1pt]      at (s2s1) {$t_{32}$};
          \node[above=2pt]      at (w0)   {$t_{31}$};
        \end{scope}
        \node at (8.8,0) {$0$};
        \path[every edge/.style=partial1,every node/.style=swap]
          (1.7,0) edge (3.3,0)
          (6.7,0) edge (8.3,0);
      \end{scope}
      \begin{scope}[shift={(0,-9.6)}]
        \begin{scope}[shift={(0,0)}]
          \pic{base3}; \pic{cond12}; \pic{cond23}; \pic{cond13};
          \node[below=2pt,gray] at (id)   {$0$};
          \node[left=1pt]       at (s1)   {$t_{21}$};
          \node[right=2pt,gray] at (s2)   {$0$};
          \node[left=1pt]       at (s1s2) {$t_{21}$};
          \node[right=1pt]      at (s2s1) {$t_{31}$};
          \node[above=2pt]      at (w0)   {$t_{31}$};
        \end{scope}
        \begin{scope}[shift={(5,0)}]
          \pic{base3}; \pic{cond12}; \pic{cond23}; \pic{cond13};
          \node[below=2pt] at (id)   {$1$};
          \node[left=2pt]  at (s1)   {$1$};
          \node[right=2pt] at (s2)   {$1$};
          \node[left=2pt]  at (s1s2) {$1$};
          \node[right=2pt] at (s2s1) {$1$};
          \node[above=2pt] at (w0)   {$1$};
        \end{scope}
        \node at (8.8,0) {$0$};
        \path[every edge/.style=partial1,every node/.style=swap]
          (1.7,0) edge (3.3,0)
          (6.7,0) edge (8.3,0);
      \end{scope}
    \end{tikzpicture}\end{equation}
\end{description}
This completes the proof of \autoref{lem:stability}.
\end{proof}

\section{Decompositions of dot representations}

Now that we have a handle on the domain and range of the divided difference operators, we can use them to decompose of some of the subrings $H_C$.

\subsection{The \texorpdfstring{$s_i$}{sᵢ}-stable case}

\begin{theorem}\label{thm:stable}
  Let $C$ be an $s_i$-stable set of divisibility conditions.
  Then, we have the internal direct sum decomposition
  \begin{equation}
    H_C = H_C^{*s_i} \oplus (x_i - x_{(i+1)}) H_C^{*s_i}
  \end{equation}
  as $\CC[t_1, \ldots, t_n]$-submodules of $H$ equipped with the dot action, where
  \begin{equation}
    H_C^{*s_i} = \set{f \in H_C \mid f * s_i = f}.
  \end{equation}
\end{theorem}

\begin{proof}
  Since $C$ is $s_i$-stable, the subgroup $\set{\id, s_i} \subseteq S_n$ of order two acts on the right on $H_C$ by the star action, and we have the natural decomposition
  \begin{equation}
    H_C = H_C^{*s_i} \oplus H_C^{*-s_i},
  \end{equation}
  where we write
  \begin{equation}
    H_C^{*-s_i} = \set{f \in H_C \mid f * s_i = -f}
  \end{equation}
  for the $(-1)$-eigenspace of the involution $f \mapsto f * s_i$.
  We claim that
  \begin{equation}
    H_C^{*-s_i} = (x_i - x_{(i+1)}) H_C^{*s_i},
  \end{equation}
  because we further claim that the maps
  \begin{equation}\begin{tikzpicture}[baseline=(a.base)]
    \node (a) at (0,0) {$H_C^{*-s_i}$};
    \node (b) at (4,0) {$H_C^{*s_i}$};
    \path[->]
    ($(a.east)+(0,3pt)$) edge node[above] {$\partial_i$} ($(b.west)+(0,3pt)$)
    ($(b.west)-(0,3pt)$) edge node[below] {$(x_i - x_{(i+1)})/2$} ($(a.east)-(0,3pt)$);
  \end{tikzpicture}\end{equation}
  form a bijection.
  This follows from the facts that
  \begin{equation}\begin{gathered}
    \partial_i(H_C^{*-s_i}) \subseteq \partial_i(H_C) \subseteq H_C^{*s_i}, \\
    (x_i - x_{(i+1)}) H_C^{*s_i} \subseteq H_C^{*-s_i}, \\
    \partial_i\paren[\big]{(x_i - x_{(i+1)}) f} / 2 = f * (1 + s_i) / 2 = f
    \quad\text{when $f \in H_C^{*s_i}$}, \\
    (x_i - x_{(i+1)}) \partial_i(f) / 2 = f * (1 - s_i) / 2 = f
    \quad\text{when $f \in H_C^{*-s_i}$},
  \end{gathered}\end{equation}
  which can be verified by direct computation.
\end{proof}

The proof of \autoref{thm:stable} yields an additional fact:
\begin{corollary}
  The kernel and image of $\partial_i : H_C \to H_C$ are equal to $H_C^{*s_i}$ when $C$ is $s_i$-stable.
\end{corollary}

\subsection{The almost-\texorpdfstring{$s_i$}{sᵢ}-stable case}

Motivated by the following decomposition, we define the notion of almost-stability. We say that a set of divisibility conditions $C$ is \emph{almost-$s_i$-stable} if $s_i \in C$ and there is a unique transposition $\tau$ such that $\tau \in C$ but $s_i \tau s_i \notin C$. If this is the case, then the sets
\begin{equation}
  C_{-} = C \setminus \set{\tau}, \qquad
  C_{+} = C \cup \set{s_i \tau s_i},
\end{equation}
obtained from $C$ by removing or adding a single condition, are $s_i$-stable.

Note that, for sets of divisibility conditions obtained from Hessenberg functions as in~\eqref{eq:hessenberg-conditions}, the notion of almost-stability corresponds to the notion of a modular triple in~\eqref{eq:modular-csf}.

\begin{theorem}\label{thm:almost-stable}
  Let $C$ be an almost-$s_i$-stable set of divisibility conditions, so that $C_{-} = C \setminus \set{\tau}$ and $C_{+} = C \cup \set{s_i \tau s_i}$ are $s_i$-stable.
  If $\tau = (i \leftrightarrow k)$, then we have the internal direct sum decomposition
  \begin{equation}
    H_C = H_{C_{+}}^{*s_i} \oplus (x_i - x_k) H_{C_{-}}^{*s_i}
  \end{equation}
  as $\CC[t_1, \ldots, t_n]$-submodules of $H$ equipped with the dot action.
  Otherwise, $\tau = (i+1 \leftrightarrow k)$ and we have the internal direct sum decomposition
  \begin{equation}
    H_C = H_{C_{+}}^{*s_i} \oplus (x_k - x_{(i+1)}) H_{C_{-}}^{*s_i}.
  \end{equation}
\end{theorem}

\begin{proof}
  We will assume that $\tau = (i \leftrightarrow k)$, since the case where $\tau = (i+1 \leftrightarrow k)$ follows by using the transformation
  \begin{equation}
    H_C * s_i = H_{s_i C s_i}.
  \end{equation}
  We claim that the decomposition is given by the following maps:
  \begin{equation}\begin{tikzpicture}[baseline=(a.base)]
    \node (a) at (0,0) {$H_{C_{+}}^{*s_i}$};
    \node (b) at (4,0) {$H_C$};
    \node (c) at (8,0) {$H_{C_{-}}^{*s_i}$};
    \path[->]
    ($(a.east)+(0,3pt)$) edge node[above] {$\id$} ($(b.west)+(0,3pt)$)
    ($(b.west)-(0,3pt)$) edge node[below] {$\partial_i \circ (x_k - x_{(i+1)})$} ($(a.east)-(0,3pt)$)
    ($(b.east)+(0,3pt)$) edge node[above] {$\partial_i$} ($(c.west)+(0,3pt)$)
    ($(c.west)-(0,3pt)$) edge node[below] {$(x_i - x_k)$} ($(b.east)-(0,3pt)$);
  \end{tikzpicture}\end{equation}
  To prove this, we need to check that:
  \begin{itemize}
    \item
      All four maps land in the stated codomain.
      Using the inclusions
      \begin{equation}\begin{gathered}
        H_{C_{+}} \subseteq H_C \subseteq H_{C_{-}} \\
        (x_i - x_k) H_{C_{-}} \subseteq H_C \\
        (x_k - x_{(i+1)}) H_C \subseteq H_{C_{+}},
      \end{gathered}\end{equation}
      we indeed have that
      \begin{equation}\begin{gathered}
        \id(H_{C_{+}}^{*s_i}) \subseteq H_{C_{+}} \subseteq H_C \\
        \partial_i(H_C) \subseteq \partial_i(H_{C_{-}}) \subseteq H_{C_{-}}^{*s_i} \\
        \partial_i\paren[\big]{(x_k - x_{(i+1)}) H_C} \subseteq \partial_i(H_{C_{+}}) \subseteq H_{C_{+}}^{*s_i} \\
        (x_i - x_k) H_{C_{-}}^{*s_i} \subseteq (x_i - x_k) H_{C_{-}} \subseteq H_C.
      \end{gathered}\end{equation}

    \item
      The two compositions $H_C \to H_C$ are complementary idempotents.
      Using the computational properties from \autoref{lem:compute}, we have
      \begin{equation}\begin{gathered}
        \partial_i\paren[\big]{(x_k - x_{(i+1)}) f} + (x_i - x_k) \partial_i(f) = f \\
        \partial_i\paren[\Big]{(x_k - x_{(i+1)}) \partial_i\paren[\big]{(x_k - x_{(i+1)}) f}} = \partial_i\paren[\big]{(x_k - x_{(i+1)}) f} \\
        (x_i - x_k) \partial_i\paren[\big]{(x_i - x_k) \partial_i(f)} = (x_i - x_k) \partial_i(f)
      \end{gathered}\end{equation}
      for every $f \in H_C$, as required.

    \item
      The compositions $H_{C_{-}}^{*s_i} \to H_{C_{-}}^{*s_i}$ and $H_{C_{+}}^{*s_i} \to H_{C_{+}}^{*s_i}$ are the identity.
      Again using \autoref{lem:compute}, we have
      \begin{equation}\begin{gathered}
        \partial_i\paren[\big]{(x_i - x_k) f} = f + (x_{(i+1)} - x_k) \partial_i(f) = f
        \quad\text{for $f \in H_{C_{-}}^{*s_i}$}, \\
        \partial_i\paren[\big]{(x_k - x_{(i+1)}) f} = f + (x_i - x_k) \partial_i(f) = f
        \quad\text{for $f \in H_{C_{+}}^{*s_i}$},
      \end{gathered}\end{equation}
      as required.
      \qedhere
  \end{itemize}
\end{proof}

\appendix\section{Some flow-up bases for \texorpdfstring{$n = 3$}{n = 3}}\label{app:flow-up}

As a further illustration of our notation and conventions, here are some examples of flow-up bases of $H_T^*(X(h)) = H_{C(h)}$ for every Hessenberg function $h : \set{1, 2, 3} \to \set{1, 2, 3}$.
We use the abbreviation $t_{ik} = (t_i - t_k)$, and indicate when $\partial_1$ or $\partial_2$ takes one basis element to another.

\vfill
\noindent
For $h = [1, 2, 3]$, we have $C(h) = \varnothing$, and $H_{C(h)} = H$ has the basis:
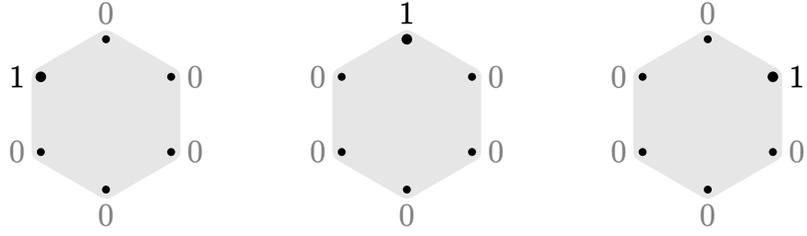
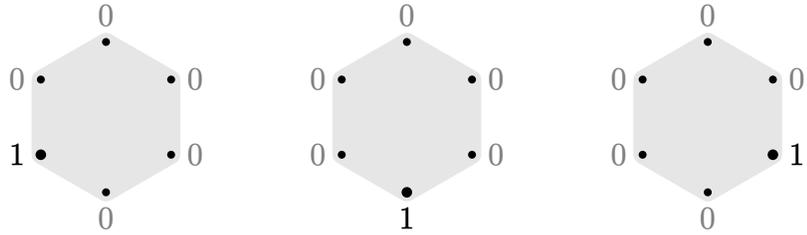
\begin{equation}\begin{tikzpicture}[baseline=0]
  \begin{scope}[shift={(0,2.2)}]
    \begin{scope}[shift={(-4,0)}]
      \pic{base3};
      \fill (s1s2) circle [radius=2pt];
      \node[below=2pt,gray] at (id)   {$0$};
      \node[left=2pt,gray]  at (s1)   {$0$};
      \node[right=2pt,gray] at (s2)   {$0$};
      \node[left=2pt]       at (s1s2) {$1$};
      \node[right=2pt,gray] at (s2s1) {$0$};
      \node[above=2pt,gray] at (w0)   {$0$};
    \end{scope}
    \begin{scope}[shift={(0,0)}]
      \pic{base3};
      \fill (w0) circle [radius=2pt];
      \node[below=2pt,gray] at (id)   {$0$};
      \node[left=2pt,gray]  at (s1)   {$0$};
      \node[right=2pt,gray] at (s2)   {$0$};
      \node[left=2pt,gray]  at (s1s2) {$0$};
      \node[right=2pt,gray] at (s2s1) {$0$};
      \node[above=2pt]      at (w0)   {$1$};
    \end{scope}
    \begin{scope}[shift={(4,0)}]
      \pic{base3};
      \fill (s2s1) circle [radius=2pt];
      \node[below=2pt,gray] at (id)   {$0$};
      \node[left=2pt,gray]  at (s1)   {$0$};
      \node[right=2pt,gray] at (s2)   {$0$};
      \node[left=2pt,gray]  at (s1s2) {$0$};
      \node[right=2pt]      at (s2s1) {$1$};
      \node[above=2pt,gray] at (w0)   {$0$};
    \end{scope}
  \end{scope}
  \begin{scope}[shift={(0,-2.2)}]
    \begin{scope}[shift={(-4,0)}]
      \pic{base3};
      \fill (s1) circle [radius=2pt];
      \node[below=2pt,gray] at (id)   {$0$};
      \node[left=2pt]       at (s1)   {$1$};
      \node[right=2pt,gray] at (s2)   {$0$};
      \node[left=2pt,gray]  at (s1s2) {$0$};
      \node[right=2pt,gray] at (s2s1) {$0$};
      \node[above=2pt,gray] at (w0)   {$0$};
    \end{scope}
    \begin{scope}[shift={(0,0)}]
      \pic{base3};
      \fill (id) circle [radius=2pt];
      \node[below=2pt]      at (id)   {$1$};
      \node[left=2pt,gray]  at (s1)   {$0$};
      \node[right=2pt,gray] at (s2)   {$0$};
      \node[left=2pt,gray]  at (s1s2) {$0$};
      \node[right=2pt,gray] at (s2s1) {$0$};
      \node[above=2pt,gray] at (w0)   {$0$};
    \end{scope}
    \begin{scope}[shift={(4,0)}]
      \pic{base3};
      \fill (s2) circle [radius=2pt];
      \node[below=2pt,gray] at (id)   {$0$};
      \node[left=2pt,gray]  at (s1)   {$0$};
      \node[right=2pt]      at (s2)   {$1$};
      \node[left=2pt,gray]  at (s1s2) {$0$};
      \node[right=2pt,gray] at (s2s1) {$0$};
      \node[above=2pt,gray] at (w0)   {$0$};
    \end{scope}
  \end{scope}
\end{tikzpicture}\end{equation}

\vfill
\noindent
For $h = [2, 2, 3]$, we have $C(h) = \set{s_1}$, and $H_{C(h)}$ has the basis:
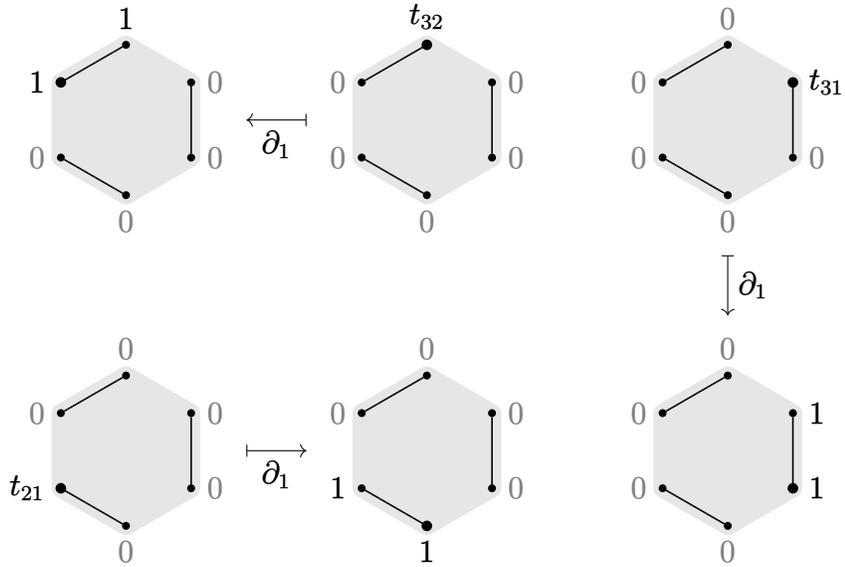
\begin{equation}\begin{tikzpicture}[baseline=0]
  \begin{scope}[shift={(0,2.2)}]
    \begin{scope}[shift={(-4,0)}]
      \pic{base3}; \pic{cond12};
      \fill (s1s2) circle [radius=2pt];
      \node[below=2pt,gray] at (id)   {$0$};
      \node[left=2pt,gray]  at (s1)   {$0$};
      \node[right=2pt,gray] at (s2)   {$0$};
      \node[left=2pt]       at (s1s2) {$1$};
      \node[right=2pt,gray] at (s2s1) {$0$};
      \node[above=2pt]      at (w0)   {$1$};
    \end{scope}
    \begin{scope}[shift={(0,0)}]
      \pic{base3}; \pic{cond12};
      \fill (w0) circle [radius=2pt];
      \node[below=2pt,gray] at (id)   {$0$};
      \node[left=2pt,gray]  at (s1)   {$0$};
      \node[right=2pt,gray] at (s2)   {$0$};
      \node[left=2pt,gray]  at (s1s2) {$0$};
      \node[right=2pt,gray] at (s2s1) {$0$};
      \node[above=2pt]      at (w0)   {$t_{32}$};
    \end{scope}
    \begin{scope}[shift={(4,0)}]
      \pic{base3}; \pic{cond12};
      \fill (s2s1) circle [radius=2pt];
      \node[below=2pt,gray] at (id)   {$0$};
      \node[left=2pt,gray]  at (s1)   {$0$};
      \node[right=2pt,gray] at (s2)   {$0$};
      \node[left=2pt,gray]  at (s1s2) {$0$};
      \node[right=2pt]      at (s2s1) {$t_{31}$};
      \node[above=2pt,gray] at (w0)   {$0$};
    \end{scope}
  \end{scope}
  \begin{scope}[shift={(0,-2.2)}]
    \begin{scope}[shift={(-4,0)}]
      \pic{base3}; \pic{cond12};
      \fill (s1) circle [radius=2pt];
      \node[below=2pt,gray] at (id)   {$0$};
      \node[left=2pt]       at (s1)   {$t_{21}$};
      \node[right=2pt,gray] at (s2)   {$0$};
      \node[left=2pt,gray]  at (s1s2) {$0$};
      \node[right=2pt,gray] at (s2s1) {$0$};
      \node[above=2pt,gray] at (w0)   {$0$};
    \end{scope}
    \begin{scope}[shift={(0,0)}]
      \pic{base3}; \pic{cond12};
      \fill (id) circle [radius=2pt];
      \node[below=2pt]      at (id)   {$1$};
      \node[left=2pt]       at (s1)   {$1$};
      \node[right=2pt,gray] at (s2)   {$0$};
      \node[left=2pt,gray]  at (s1s2) {$0$};
      \node[right=2pt,gray] at (s2s1) {$0$};
      \node[above=2pt,gray] at (w0)   {$0$};
    \end{scope}
    \begin{scope}[shift={(4,0)}]
      \pic{base3}; \pic{cond12};
      \fill (s2) circle [radius=2pt];
      \node[below=2pt,gray] at (id)   {$0$};
      \node[left=2pt,gray]  at (s1)   {$0$};
      \node[right=2pt]      at (s2)   {$1$};
      \node[left=2pt,gray]  at (s1s2) {$0$};
      \node[right=2pt]      at (s2s1) {$1$};
      \node[above=2pt,gray] at (w0)   {$0$};
    \end{scope}
  \end{scope}
  \path
  (-1.6, 2.2) edge[partial1]      (-2.4, 2.2)
  ( 4  , 0.4) edge[partial1]      ( 4  ,-0.4)
  (-2.4,-2.2) edge[partial1,swap] (-1.6,-2.2)
  ;
\end{tikzpicture}\end{equation}

\vfill
\noindent
For $h = [1, 3, 3]$, we have $C(h) = \set{s_2}$, and $H_{C(h)}$ has the basis:
\begin{equation}\hspace{-1ex}\begin{tikzpicture}[baseline=0]
  \begin{scope}[shift={(0,2.2)}]
    \begin{scope}[shift={(-4,0)}]
      \pic{base3}; \pic{cond23};
      \fill (s1s2) circle [radius=2pt];
      \node[below=2pt,gray] at (id)   {$0$};
      \node[left=2pt,gray]  at (s1)   {$0$};
      \node[right=2pt,gray] at (s2)   {$0$};
      \node[left=2pt]       at (s1s2) {$t_{31}$};
      \node[right=2pt,gray] at (s2s1) {$0$};
      \node[above=2pt,gray] at (w0)   {$0$};
    \end{scope}
    \begin{scope}[shift={(0,0)}]
      \pic{base3}; \pic{cond23};
      \fill (w0) circle [radius=2pt];
      \node[below=2pt,gray] at (id)   {$0$};
      \node[left=2pt,gray]  at (s1)   {$0$};
      \node[right=2pt,gray] at (s2)   {$0$};
      \node[left=2pt,gray]  at (s1s2) {$0$};
      \node[right=2pt,gray] at (s2s1) {$0$};
      \node[above=2pt]      at (w0)   {$t_{21}$};
    \end{scope}
    \begin{scope}[shift={(4,0)}]
      \pic{base3}; \pic{cond23};
      \fill (s2s1) circle [radius=2pt];
      \node[below=2pt,gray] at (id)   {$0$};
      \node[left=2pt,gray]  at (s1)   {$0$};
      \node[right=2pt,gray] at (s2)   {$0$};
      \node[left=2pt,gray]  at (s1s2) {$0$};
      \node[right=2pt]      at (s2s1) {$1$};
      \node[above=2pt]      at (w0)   {$1$};
    \end{scope}
  \end{scope}
  \begin{scope}[shift={(0,-2.2)}]
    \begin{scope}[shift={(-4,0)}]
      \pic{base3}; \pic{cond23};
      \fill (s1) circle [radius=2pt];
      \node[below=2pt,gray] at (id)   {$0$};
      \node[left=2pt]       at (s1)   {$1$};
      \node[right=2pt,gray] at (s2)   {$0$};
      \node[left=2pt]       at (s1s2) {$1$};
      \node[right=2pt,gray] at (s2s1) {$0$};
      \node[above=2pt,gray] at (w0)   {$0$};
    \end{scope}
    \begin{scope}[shift={(0,0)}]
      \pic{base3}; \pic{cond23};
      \fill (id) circle [radius=2pt];
      \node[below=2pt]      at (id)   {$1$};
      \node[left=2pt,gray]  at (s1)   {$0$};
      \node[right=2pt]      at (s2)   {$1$};
      \node[left=2pt,gray]  at (s1s2) {$0$};
      \node[right=2pt,gray] at (s2s1) {$0$};
      \node[above=2pt,gray] at (w0)   {$0$};
    \end{scope}
    \begin{scope}[shift={(4,0)}]
      \pic{base3}; \pic{cond23};
      \fill (s2) circle [radius=2pt];
      \node[below=2pt,gray] at (id)   {$0$};
      \node[left=2pt,gray]  at (s1)   {$0$};
      \node[right=2pt]      at (s2)   {$t_{32}$};
      \node[left=2pt,gray]  at (s1s2) {$0$};
      \node[right=2pt,gray] at (s2s1) {$0$};
      \node[above=2pt,gray] at (w0)   {$0$};
    \end{scope}
  \end{scope}
  \path
  ( 1.6, 2.2) edge[partial2,swap] ( 2.4, 2.2)
  (-4  , 0.4) edge[partial2,swap] (-4  ,-0.4)
  ( 2.4,-2.2) edge[partial2]      ( 1.6,-2.2)
  ;
\end{tikzpicture}\hspace{-1ex}\end{equation}

\vfill
\noindent
For $h = [2, 3, 3]$, we have $C(h) = \set{s_1, s_2}$, and $H_{C(h)}$ has the basis:
\begin{equation}\hspace{-1ex}\begin{tikzpicture}[baseline=0]
  \begin{scope}[shift={(0,2.2)}]
    \begin{scope}[shift={(-4,0)}]
      \pic{base3}; \pic{cond12}; \pic{cond23};
      \fill (s1s2) circle [radius=2pt];
      \node[below=2pt,gray] at (id)   {$0$};
      \node[left=2pt,gray]  at (s1)   {$0$};
      \node[right=2pt,gray] at (s2)   {$0$};
      \node[left=2pt]       at (s1s2) {$t_{31}$};
      \node[right=2pt,gray] at (s2s1) {$0$};
      \node[above=2pt]      at (w0)   {$t_{21}$};
    \end{scope}
    \begin{scope}[shift={(0,0)}]
      \pic{base3}; \pic{cond12}; \pic{cond23};
      \fill (w0) circle [radius=2pt];
      \node[below=2pt,gray] at (id)   {$0$};
      \node[left=2pt,gray]  at (s1)   {$0$};
      \node[right=2pt,gray] at (s2)   {$0$};
      \node[left=2pt,gray]  at (s1s2) {$0$};
      \node[right=2pt,gray] at (s2s1) {$0$};
      \node[above=2pt]      at (w0)   {$t_{21} t_{32}$};
    \end{scope}
    \begin{scope}[shift={(4,0)}]
      \pic{base3}; \pic{cond12}; \pic{cond23};
      \fill (s2s1) circle [radius=2pt];
      \node[below=2pt,gray] at (id)   {$0$};
      \node[left=2pt,gray]  at (s1)   {$0$};
      \node[right=2pt,gray] at (s2)   {$0$};
      \node[left=2pt,gray]  at (s1s2) {$0$};
      \node[right=2pt]      at (s2s1) {$t_{31}$};
      \node[above=2pt]      at (w0)   {$t_{32}$};
    \end{scope}
  \end{scope}
  \begin{scope}[shift={(0,-2.2)}]
    \begin{scope}[shift={(-4,0)}]
      \pic{base3}; \pic{cond12}; \pic{cond23};
      \fill (s1) circle [radius=2pt];
      \node[below=2pt,gray] at (id)   {$0$};
      \node[left=2pt]       at (s1)   {$t_{21}$};
      \node[right=2pt,gray] at (s2)   {$0$};
      \node[left=2pt]       at (s1s2) {$t_{23}$};
      \node[right=2pt,gray] at (s2s1) {$0$};
      \node[above=2pt,gray] at (w0)   {$0$};
    \end{scope}
    \begin{scope}[shift={(0,0)}]
      \pic{base3}; \pic{cond12}; \pic{cond23};
      \fill (id) circle [radius=2pt];
      \node[below=2pt]      at (id)   {$1$};
      \node[left=2pt]       at (s1)   {$1$};
      \node[right=2pt]      at (s2)   {$1$};
      \node[left=2pt]       at (s1s2) {$1$};
      \node[right=2pt]      at (s2s1) {$1$};
      \node[above=2pt]      at (w0)   {$1$};
    \end{scope}
    \begin{scope}[shift={(4,0)}]
      \pic{base3}; \pic{cond12}; \pic{cond23};
      \fill (s2) circle [radius=2pt];
      \node[below=2pt,gray] at (id)   {$0$};
      \node[left=2pt,gray]  at (s1)   {$0$};
      \node[right=2pt]      at (s2)   {$t_{32}$};
      \node[left=2pt,gray]  at (s1s2) {$0$};
      \node[right=2pt]      at (s2s1) {$t_{12}$};
      \node[above=2pt,gray] at (w0)   {$0$};
    \end{scope}
  \end{scope}
  \path
  (-1.6, 2.2) edge[partial1]      (-2.4, 2.2)
  ( 1.6, 2.2) edge[partial2,swap] ( 2.4, 2.2)
  ;
\end{tikzpicture}\hspace{-1ex}\end{equation}

\vfill
\noindent
For $h = [3, 3, 3]$, we have $C(h) = \set{s_1, s_2, w_0}$, and $H_{C(h)}$ has the basis:
\begin{equation}\hspace{-4ex}\begin{tikzpicture}[baseline=0]
  \begin{scope}[shift={(0,2.2)}]
    \begin{scope}[shift={(-4,0)}]
      \pic{base3}; \pic{cond12}; \pic{cond23}; \pic{cond13};
      \fill (s1s2) circle [radius=2pt];
      \node[below=2pt,gray] at (id)   {$0$};
      \node[left=2pt,gray]  at (s1)   {$0$};
      \node[right=2pt,gray] at (s2)   {$0$};
      \node[left=2pt]       at (s1s2) {$t_{21} t_{31}$};
      \node[right=2pt,gray] at (s2s1) {$0$};
      \node[above=2pt]      at (w0)   {$t_{21} t_{31}$};
    \end{scope}
    \begin{scope}[shift={(0,0)}]
      \pic{base3}; \pic{cond12}; \pic{cond23}; \pic{cond13};
      \fill (w0) circle [radius=2pt];
      \node[below=2pt,gray] at (id)   {$0$};
      \node[left=2pt,gray]  at (s1)   {$0$};
      \node[right=2pt,gray] at (s2)   {$0$};
      \node[left=2pt,gray]  at (s1s2) {$0$};
      \node[right=2pt,gray] at (s2s1) {$0$};
      \node[above=2pt]      at (w0)   {$t_{21} t_{31} t_{32}$};
    \end{scope}
    \begin{scope}[shift={(4,0)}]
      \pic{base3}; \pic{cond12}; \pic{cond23}; \pic{cond13};
      \fill (s2s1) circle [radius=2pt];
      \node[below=2pt,gray] at (id)   {$0$};
      \node[left=2pt,gray]  at (s1)   {$0$};
      \node[right=2pt,gray] at (s2)   {$0$};
      \node[left=2pt,gray]  at (s1s2) {$0$};
      \node[right=2pt]      at (s2s1) {$t_{31} t_{32}$};
      \node[above=2pt]      at (w0)   {$t_{31} t_{32}$};
    \end{scope}
  \end{scope}
  \begin{scope}[shift={(0,-2.2)}]
    \begin{scope}[shift={(-4,0)}]
      \pic{base3}; \pic{cond12}; \pic{cond23}; \pic{cond13};
      \fill (s1) circle [radius=2pt];
      \node[below=2pt,gray] at (id)   {$0$};
      \node[left=2pt]       at (s1)   {$t_{21}$};
      \node[right=2pt,gray] at (s2)   {$0$};
      \node[left=2pt]       at (s1s2) {$t_{21}$};
      \node[right=2pt]      at (s2s1) {$t_{31}$};
      \node[above=2pt]      at (w0)   {$t_{31}$};
    \end{scope}
    \begin{scope}[shift={(0,0)}]
      \pic{base3}; \pic{cond12}; \pic{cond23}; \pic{cond13};
      \fill (id) circle [radius=2pt];
      \node[below=2pt]      at (id)   {$1$};
      \node[left=2pt]       at (s1)   {$1$};
      \node[right=2pt]      at (s2)   {$1$};
      \node[left=2pt]       at (s1s2) {$1$};
      \node[right=2pt]      at (s2s1) {$1$};
      \node[above=2pt]      at (w0)   {$1$};
    \end{scope}
    \begin{scope}[shift={(4,0)}]
      \pic{base3}; \pic{cond12}; \pic{cond23}; \pic{cond13};
      \fill (s2) circle [radius=2pt];
      \node[below=2pt,gray] at (id)   {$0$};
      \node[left=2pt,gray]  at (s1)   {$0$};
      \node[right=2pt]      at (s2)   {$t_{32}$};
      \node[left=2pt]       at (s1s2) {$t_{31}$};
      \node[right=2pt]      at (s2s1) {$t_{32}$};
      \node[above=2pt]      at (w0)   {$t_{31}$};
    \end{scope}
  \end{scope}
  \path
  (-1.6, 2.2) edge[partial1]      (-2.4, 2.2)
  ( 4  , 0.4) edge[partial1]      ( 4  ,-0.4)
  (-2.4,-2.2) edge[partial1,swap] (-1.6,-2.2)
  ( 1.6, 2.2) edge[partial2,swap] ( 2.4, 2.2)
  (-4  , 0.4) edge[partial2,swap] (-4  ,-0.4)
  ( 2.4,-2.2) edge[partial2]      ( 1.6,-2.2)
  ;
\end{tikzpicture}\hspace{-4ex}\end{equation}


\setlength{\hfuzz}{5pt}  
\bigskip{\footnotesize\par\noindent\textsc{%
LACIM,
Université du Québec à Montréal,
201 Avenue du Président-Kennedy,
Montréal QC\hspace{.7em}H2X~3Y7,
Canada}
\par\noindent\textit{Email address}:
\mailto{mathieu.guaypaquet@lacim.ca}}

\end{document}